\documentclass[12pt, reqno]{amsart}
\usepackage{amsmath, amsthm, amscd, amsfonts, amssymb, graphicx, color}
\usepackage[bookmarksnumbered, plainpages]{hyperref}
\usepackage{amssymb}
\hypersetup{colorlinks=true,linkcolor=red, anchorcolor=green,
citecolor=cyan, urlcolor=red, filecolor=magenta, pdftoolbar=true}

\newtheorem{thm}{Theorem}[section]
\newtheorem{cor}[thm]{Corollary}
\newtheorem{lemma}[thm]{Lemma}

\newtheorem{prob}[thm]{Problem}
\newtheorem{example}[thm]{Example}

\theoremstyle{definition}
\newtheorem{definition}{Definition}
\newtheorem{rem}[definition]{Remark}

\numberwithin{equation}{section}

\begin{document}

\title{Some remarks on orthogonally additive operators on Riesz spaces}

\author{Olena Fotiy}
\address{Department of Mathematics and Informatics\\
Chernivtsi National University (Ukraine)}

\email{ofotiy@ukr.net}

\author{Vladimir Kadets}

\address{Department of Mathematics Holon Institute of Technology (Israel) and V.N.Karazin Kharkiv National University (Ukraine) }

\email{vova1kadets@yahoo.com}

\author{Mikhail Popov}

\address{Institute of Exact and Technical Sciences,
Pomeranian University in S{\l}upsk, S{\l}upsk (Poland) and
Vasyl Stefanyk Precarpathian National University, Ivano-Frankivsk (Ukraine)}

\email{misham.popov@gmail.com}



\begin{abstract}
We study orthogonally additive operators between Riesz spaces without the Dedekind completeness assumption on the range space. Our first result gives necessary and sufficient conditions on a pair of Riesz spaces $(E,F)$ for which every orthogonally additive operator from $E$ to $F$ is laterally-to-order bounded. Second result provides sufficient conditions on a pair of orthogonally additive operators $S$ and $T$ to have $S \vee T$, as well as to have $S \wedge T$, and consequently, for an orthogonally additive operator $T$ to have $T^+$, $T^-$ or $|T|$ without any assumption on the domain and range spaces. Finally we prove an analogue of Meyer's theorem on the existence of modules of disjointness preserving operator for the setting of orthogonally additive operators.
\end{abstract}

\maketitle

\section{Introduction}

Orthogonally additive operators (OAOs in short) between Riesz spaces generalize linear operators. Nonlinear OAOs naturally appeared in different fields of analysis, including Uryson and Nemytsky integral operators. Fortunately, lots of tools applied for linear operators, which use the additivity only on disjoint vectors, maintain applicability to OAOs (see \cite{PPSurv} for a recent survey on OAOs). For familiarly used notions and facts on Riesz spaces we refer the reader to \cite{ABu}.

Let $E,F$ be Riesz spaces. A function $T \colon E \to F$ is called an \emph{OAO} provided $T(x + y) = T(x) + T(y)$ for all disjoint vectors $x,y \in E$. An OAO $T \colon E \to F$ is said to be
\begin{itemize}
  \item \emph{positive}\index{posOAO} (write $T \ge 0$) provided $T(x) \ge 0$ for all $x \in E$;
  \item \emph{order bounded}\index{ordboundOAO} or an \emph{abstract Uryson operator} provided $T$ sends order bounded subsets of $E$ to order bounded subsets of $F$.
\end{itemize}
The symbols $\mathcal{O}(E,F)$, $\mathcal{O}^+(E,F)$ and $\mathcal{U}(E,F)$ stand to denote the sets of all OAOs, the set of all positive OAOs and the set of all abstract Uryson operators respectively.

Obviously, $\mathcal{O}(E,F)$ is an ordered vector space with respect to the order $S \le T$ provided $T - S \ge 0$, that is, $S(x) \le T(x)$ for all $x \in E$. Remark that the standard order on the vector space $\mathcal{L}(E,F)$ of all linear operators from $E$ to $F$, which is a linear subspace of $\mathcal{O}(E,F)$, is different, and the only linear operator which is positive as an OAO is zero.

Let $E$ be a Riesz space. The notation $x = y \sqcup z$ for $x,y,z \in E$ means that $x = y + z$ and $y \perp z$. An element $x \in E$ is called a \emph{fragment}\footnote{component in the terminology of \cite{ABu}} of $y \in E$ provided $x \perp y-x$. In this case we write $x \pmb{\sqsubseteq} y$. The binary relation $\pmb{\sqsubseteq}$ is a (non-strict) partial order on $E$, called the lateral order. The set of all fragments of a given element $e \in E$ will be denoted by $\mathfrak{F}_e$. The set $\mathfrak{F}_e$ is a Boolean algebra with zero $0$, unit $e$ with respect to the operations $\pmb{\cup}$ and $\pmb{\cap}$ of taking the lateral supremum and infimum respectively. Moreover, $x \pmb{\cup} y = (x^+ \vee y^+) - (x^- \vee y^-)$ and $x \pmb{\cap} y = (x^+ \wedge y^+) - (x^- \wedge y^-)$ for all $x,y \in \mathfrak{F}_e$ (see \cite{MPP} for details).

Let $E,F$ be Riesz spaces. An OAO $T \colon E \to F$ is said to be \emph{laterally-to-order bounded} provided for every $e \in E$ the image $T(\mathfrak{F}_e)$ under $T$ of the set $\mathfrak{F}_e$ of all fragments of $e$ is order bounded in $F$. The set of laterally-to-order bounded OAOs from $E$ to $F$ is denoted by $\mathcal{P}(E,F)$. The notion of laterally-to-order bounded OAOs was introduced and first studied by Pliev and Ramdane in \cite{PlRa}.

Obviously, $\mathcal{U}(E,F) \subseteq \mathcal{P}(E,F) \subseteq \mathcal{O}(E,F)$.

\section{Laterally-to-order bounded OAOs}

The first study of the Riesz space structure of the ordered vector space $\mathcal{O}(E,F)$ of all OAOs between Riesz spaces $E$ and $F$ was presented by Maz\'{o}n and Segura de Le\'{o}n in \cite{Maz-1} and \cite{Maz-2}(1990). One of the main structural results in \cite{Maz-1} asserts that, if the Riesz space $F$ is Dedekind complete then $\mathcal{U}(E,F)$ is a Dedekind complete Riesz space as well, and some natural formulas for lattice operations on $\mathcal{U}(E,F)$ were provided. In 2018 Pliev and Ramdane \cite{PlRa} generalized the above mentioned result of Maz\'{o}n and Segura de Le\'{o}n to the class $\mathcal{P}(E,F)$ of laterally-to-order bounded OAOs; see the next theorem.

\begin{thm}[\cite{PlRa}] \label{th:PlRamd}
Let $E,F$ be Riesz spaces with $F$ Dedekind complete. Then $\mathcal{P}(E,F)$ is a Dedekind complete Riesz space and the following assertions hold.
\begin{enumerate}
  \item For every $S, T \in \mathcal{P}(E,F)$ and every $x \in E$ one has
\begin{enumerate}
  \item $(S \vee T)(x) = \sup \{S(u) + T(v): \, x = u \sqcup v\}$;
  \item $(S \wedge T)(x) \, = \, \inf \{S(u) + T(v): \, x = u \sqcup v\}$;
  \item $T^+(x) = \sup \{T(u) : \, u \sqsubseteq x\}$;
  \item $T^-(x) = - \inf \{T(u) : \, u \sqsubseteq x\}$;
  \item $|T(x)| \le |T|(x)$.
\end{enumerate}
  \item The set $\mathcal{U}(E,F)$ is an order ideal of $\mathcal{P}(E,F)$ and hence itself is a Dedekind complete Riesz space with properties {\rm (a)-(e)}.
\end{enumerate}
\end{thm}

It is interesting to note that formulas (a) and (b) coincide with Abramovich's formulas for join and meet of linear operators, which are true under the additional assumption on $E$ to have the principal projection property and false without this assumption \cite[Theorem\,1.50]{ABu}.

So the Riesz space $\mathcal{P}(E,F)$ is a much more natural object for the study of OAOs than $\mathcal{U}(E,F)$. Remark that $\mathcal{U}(E,F)$ need not be a band of $\mathcal{P}(E,F)$ \cite{PlRa}.

The linear subspace $\mathcal{P}(E,F)$ of $\mathcal{O}(E,F)$ is so ``large'' that these two ordered vector spaces have common positive cone, that is, every positive OAO is laterally-to-order bounded, which is easy to see. However, the inclusion $\mathcal{P}(E,F) \subset \mathcal{O}(E,F)$ can be strict due to the following example.

\begin{example}[\cite{PlRa}, Example\,3.4] \label{ex:plram}
Let $\ell_0^\infty$ be the Riesz space of all real eventually constant sequences $x = (x_n)_{n=1}^\infty$, that is, $(\exists k \in \mathbb N)(\forall n \ge k) \, x_n = x_k$ with the coordinate-wise order. Then the map $T \colon E \to \mathbb R$ defined by setting $T(x) = \sum_{n=1}^\infty \frac{(-1)^n |x_n|}{n}$ for an arbitrary $x \in E$, belongs to $\mathcal{O}(E,\mathbb R) \setminus \mathcal{P}(E,\mathbb R)$.
\end{example}

Observe, that the Riesz space $\ell_0^\infty$ is not Dedekind complete, and the idea used in Example\,\ref{ex:plram} cannot help to construct an example of the kind with Dedekind complete domain space.

The following result characterizes pairs of Riesz spaces $(E,F)$ for which the inclusion $\mathcal{P}(E,F) \subset \mathcal{O}(E,F)$ is strict.

\begin{thm} \label{th:knbdbj}
Let $E,F$ be Riesz spaces with $F$ Archimedean. Then the following assertions are equivalent.
\begin{enumerate}
  \item $\mathcal{P}(E,F) = \mathcal{O}(E,F)$.
  \item Either $F = \{0\}$ or for every $e \in E$ the set $\mathfrak{F}_e$ is finite.
\end{enumerate}
\end{thm}

\begin{proof}
(1) $\Rightarrow$ (2). Assuming the contrary, let $f \in F \setminus \{0\}$ and let $e$ be an element of $E$ with infinite $\mathfrak{F}_e$. Our goal is to construct an OAO $T \colon E \to F$ which is not laterally-to-order bounded (it is going to be even linear). Since $\mathfrak{F}_e$ is an infinite Boolean algebra, we can choose a disjoint sequence $(e_n)_{n=1}^\infty$ of nonzero elements of $\mathfrak{F}_e$, which is linearly independent in $E$. Extend $(e_n)_{n=1}^\infty$ to a Hamel basis $(e_i)_{i \in I}$, $\mathbb N \subseteq I$ in the linear space $E$. Then define a linear operator $T \colon E \to F$ using the Hamel basis as follows. First define $Te_n = n f$ for all $n \in \mathbb N$ and somehow define $T$ on the rest of the basis, say, $Te_i = 0$ for $i \in I \setminus \mathbb N$. Finally we extend $T$ from the basis to the entire $E$ by linearity. It remains to observe that $T(\mathfrak{F}_e)$ is not order bounded in $F$, because $nf \in T(\mathfrak{F}_e)$ for all $n \in \mathbb N$ and $F$ is Archimedean.

(2) $\Rightarrow$ (1) is obvious.
\end{proof}

So the inclusion $\mathcal{P}(E,F) \subset \mathcal{O}(E,F)$ is strict, except for somewhat trivial cases.

\begin{rem}
The infinite dimensional Dedekind complete Riesz space $c_{00}$ of all eventually zero sequences possesses the second part of (2) in Theorem\,\ref{th:knbdbj}, and so $\mathcal{P}(c_{00},F) = \mathcal{O}(c_{00},F)$ for every Archimedean Riesz space $F$.
\end{rem}

\section{Existence of a supremum $S \vee T$ for OAOs $S$ and $T$}

In this section, we generalize Pliev and Ramdane's Theorem\,\ref{th:PlRamd} and find sufficient condition on OAOs $S$ and $T$ to have $S \vee T$ as well as to have $S \wedge T$, and consequently, for an OAO $T$ to have $T^+$, $T^-$ or $|T|$ without any assumption on the domain and range spaces.

We need the following Pliev Lemma \cite[Proposition\,3.11]{Pl}, which is a lateral analogue of the Riesz decomposition property \cite[Theorem\,1.20]{ABu}.

\begin{lemma}[\cite{Pl}] \label{le:Pli}
Let $E$ be a Riesz space, $u_1, \ldots, u_m, v_1, \ldots v_n \in E$ and $e := \bigsqcup_{i=1}^m u_i = \bigsqcup_{k=1}^n v_k$. Then there exists a (disjoint) double sequence ${{(w_{i,k})}_{i=1}^m}_{k=1}^n$ in $E$ such that
\begin{enumerate}
  \item $u_i = \bigsqcup_{k=1}^n w_{i,k}$ for any $i \in \{1, \ldots, m\}$;
  \item $v_k = \bigsqcup_{i=1}^m w_{i,k}$ for any  $k \in \{1, \ldots, n\}$.
\end{enumerate}
\end{lemma}

\begin{proof}[Sketch of proof of Lemma\,\ref{le:Pli}]
Since $u_i, v_k$ are elements of the Boolean algebra $\mathfrak{F}_e$, the elements $w_{i,k} := u_i \pmb{\cap} v_k$ are well defined and possess the desired properties.
\end{proof}

\begin{thm} \label{th:veekfuy}
Let $E, F$ be Riesz spaces and $S,T \in \mathcal{O}(E,F)$. If for every $x \in E$ the supremum $R(x) := \sup \{S(u) + T(v): \, x = u \sqcup v\}$ exists in $F$ then $S \vee T$ exists in $\mathcal{O}(E,F)$ and $S \vee T = R$. Moreover, in case of existence the following hold.
\begin{enumerate}
  \item If $S,T \in \mathcal{P}(E,F)$ then $S \vee T \in \mathcal{P}(E,F)$.
  \item If $S,T \in \mathcal{U}(E,F)$ then $S \vee T \in \mathcal{U}(E,F)$.
\end{enumerate}
\end{thm}

\begin{proof}
Show that $R$ is an OAO. Fix any $x,y \in E$ with $x \perp y$. Let $x + y = u \sqcup v$ be any decomposition into disjoint fragments. Choose by Lemma\,\ref{le:Pli} disjoint vectors $w_i$, $i = 1, \ldots, 4$ so that $x = w_1 \sqcup w_3$, $y = w_2 \sqcup w_4$, $u = w_1 \sqcup w_2$ and $v = w_3 \sqcup w_4$. Then
$$
S(u) + T(v) = S(w_1) + S(w_2) + T(w_3) + T(w_4) \le R(x) + R(y).
$$
By the arbitrariness of a decomposition of $x + y$, we obtain $R(x+y) \le R(x) + R(y)$. On the other hand,
\begin{align*}
R(x) + R(y) &= \sup \{S(u) + T(v): \, x = u \sqcup v\} + \sup \{S(w) + T(z): \, y = w \sqcup z\} \\ &= \sup \{S(u+w) + T(v+z): \, x = u \sqcup v, \, y = w \sqcup z\} \\
&\le R(x+y).
\end{align*}
Thus, $R$ is an OAO. Show that $R = S \vee T$ in $\mathcal{O}(E,F)$. Indeed, obviously $S \le R$ and $T \le R$. Let $P \in \mathcal{O}(E,F)$ satisfy $S \le P$ and $T \le P$. Fix any $x \in E$ and consider any decomposition $x = u \sqcup v$. Then $S(u) + T(v) \le P(u) + T(v) = P(x)$. By the arbitrariness of the decomposition of $x$, $R(x) \le P(x)$ and so $R \le P$. Thus, $R = S \vee T$.

Items (1) and (2) easily follow from the obtained formula for $S \vee T$.
\end{proof}

Theorem\,\ref{th:veekfuy} gives the following consequences for the existence of meet of two OAOs, positive and negative parts and modulus of an OAO.

\begin{cor} \label{cor:hfhn1}
Let $E, F$ be Riesz spaces and $S,T \in \mathcal{O}(E,F)$. If for every $x \in E$ the infimum $R(x) := \inf \{S(u) + T(v): \, x = u \sqcup v\}$ exists in $F$ then $S \wedge T$ exists in $\mathcal{O}(E,F)$ and $S \wedge T = R$. Moreover, in case of existence the following hold.
\begin{enumerate}
  \item If $S,T \in \mathcal{P}(E,F)$ then $S \wedge T \in \mathcal{P}(E,F)$.
  \item If $S,T \in \mathcal{U}(E,F)$ then $S \wedge T \in \mathcal{U}(E,F)$.
\end{enumerate}
\end{cor}

\begin{proof}
For every $x \in E$ one has
\begin{align*}
(S \wedge T)(x) = - \bigl( (-S) \vee (-T) \bigr)(x) &= - \sup \{ -S(u) -T(v): \, x = u \sqcup v\} \,\,\,\,\,\,\,\,\,\,\,\,\\
&= \inf \{S(u) + T(v): \, x = u \sqcup v\}.
\end{align*}
Items (1) and (2) follow from the obtained formula for $S \wedge T$.
\end{proof}

Next three corollaries directly follow from Theorem\,\ref{th:veekfuy}.

\begin{cor} \label{cor:hfhn2}
Let $E, F$ be Riesz spaces and $T \in \mathcal{O}(E,F)$. If for every $x \in E$ the supremum $R(x) := \sup \{T(u): \, u \sqsubseteq x\}$ exists in $F$ then $T^+$ exists in $\mathcal{O}(E,F)$ and $T^+ = R$. Moreover, in case of existence the following hold.
\begin{enumerate}
  \item If $T \in \mathcal{P}(E,F)$ then $T^+ \in \mathcal{P}(E,F)$.
  \item If $T \in \mathcal{U}(E,F)$ then $T^+ \in \mathcal{U}(E,F)$.
\end{enumerate}
\end{cor}

\begin{cor} \label{cor:hfhn3}
Let $E, F$ be Riesz spaces and $T \in \mathcal{O}(E,F)$. If for every $x \in E$ the infimum $- R(x) := \inf \{T(u): \, u \sqsubseteq x\}$ exists in $F$ then $T^-$ exists in $\mathcal{O}(E,F)$ and $T^- = R$. Moreover, in case of existence the following hold.
\begin{enumerate}
  \item If $T \in \mathcal{P}(E,F)$ then $T^- \in \mathcal{P}(E,F)$.
  \item If $T \in \mathcal{U}(E,F)$ then $T^- \in \mathcal{U}(E,F)$.
\end{enumerate}
\end{cor}

\begin{cor} \label{cor:hfhn4}
Let $E, F$ be Riesz spaces and $T \in \mathcal{O}(E,F)$. If for every $x \in E$ the supremum $R(x) := \sup \{T(u) - T(v): \, x = u \sqcup v\}$ exists in $F$ then $|T|$ exists in $\mathcal{O}(E,F)$ and $|T| = R$. Moreover, in case of existence the following hold.
\begin{enumerate}
  \item If $T \in \mathcal{P}(E,F)$ then $|T| \in \mathcal{P}(E,F)$.
  \item If $T \in \mathcal{U}(E,F)$ then $|T| \in \mathcal{U}(E,F)$.
\end{enumerate}
\end{cor}

Remark that Theorem\,\ref{th:veekfuy} together with all its corollaries entirely yield Pliev-Ramdane Theorem\,\ref{th:PlRamd}.

\section{The existence of modules of a disjointness preserving OAO}

As usual, $\mathcal L_b (E,F)$ denotes the ordered vector space of all order bounded linear operators between Riesz spaces $E$ and $F$, which is a Riesz space once $F$ is Dedekind complete \cite[Theorem\,1.18]{ABu}. Recall that an OAO $T \colon E \to F$ is said to preserve disjointness provided $T(x) \perp T(y)$ for all $x,y \in E$ with $x \perp y$.

In this section, we show that the following Meyer theorem on the existence of a modulus $|T|$ of a disjointness preserving linear operator $T$ holds true for OAOs.

\begin{thm}[Theorem\,2.40, \cite{ABu}] \label{th:mayers}
Let $E, F$ be Riesz spaces with $F$ Archimedean. Then every disjointness preserving linear operator $T \in \mathcal L_b (E,F)$ has the modulus $|T| := T \vee (-T)$. Moreover, for every $x \in E$
$$
\bigl|T \bigr| |x| = \bigl| T |x| \bigr| = \bigl| Tx \bigr|.
$$
\end{thm}

We also prove an analogue of Meyer's lemma on disjointness preserving operators \cite[Lemma\,2.39]{ABu} for the setting of OAOs.

The proof of the next theorem, which is an analogue of Meyer's theorem for OAOs, based on the results of the previous section, is much more simpler than the proof of the original Meyer Theorem\,\ref{th:mayers} for linear operators. Although OAOs generalize linear operators, we cannot consider our theorem as a generalization of Meyer's theorem for linear operators, because spaces of linear operators and OAOs have different orders.

\begin{thm} \label{th:kjhhfy7dy}
Let $E, F$ be Riesz spaces and let $T \in \mathcal O (E,F)$ preserve disjointness. Then the following assertions hold.
\begin{enumerate}
  \item $T \in \mathcal{P} (E,F)$.
  \item $|T|$ exists in $\mathcal O (E,F)$ and $|T| \in \mathcal{P} (E,F)$. Moreover, for every $x \in E$ one has $
\bigl|T \bigr| (x) = \bigl| T (x) \bigr|.$
  \item $T$ is regular; $T^+$ and $T^-$ exist and both preserve disjointness. Moreover, for every $x \in E$ one has $T^+(x) = \bigl( T(x) \bigr)^+$ and $T^+(x) = \bigl( T(x) \bigr)^+.$
  \item Let $\{x,y\} \subset E$ be a laterally bounded set. Then $\bigl( T (x) \bigr)^+ \wedge \bigl( T (y) \bigr)^- = 0$.
\end{enumerate}
\end{thm}

Remark that Theorem\,\ref{th:kjhhfy7dy} was partially proved in \cite[Lemma\,2.2]{PlPol} under the additional assumptions of Dedekind completeness of $F$, and formulas in (2) and (3) for the modulus, positive and negative parts were discovered earlier for an order bounded $T$ in \cite[Lemma\,3.1]{AP}. Item (4) is an analogue of Meyer's lemma (see Lemma\,\ref{th:mayers} below).

For the proof of Theorem\,\ref{th:kjhhfy7dy}, we need the following two known simple facts.

\begin{lemma}[\cite{MPP}] \label{th:latbound}
Let $E, F$ be Riesz spaces and $T \colon E \to F$ a disjointness preserving OAO. Then for every $e \in E$ and every $x \sqsubseteq e$ one has $T(x) \sqsubseteq T(e)$.
\end{lemma}

The proof is straightforward and easy. Note that the converse implication also holds: every lateral order preserving OAO preserves disjointness \cite[Theorem\,4.9]{MPP}.

\begin{lemma}\cite{MPP} \label{th:jgnbfgy}
Let $E$ be a Riesz spaces and $x,y \in E$.
\begin{enumerate}
  \item If $x \sqsubseteq y$ then $x^+ \sqsubseteq y^+$, $x^- \sqsubseteq y^-$, $|x| \sqsubseteq |y|$.
  \item If $|x| \sqsubseteq |y|$ then $|x| \le |y|$.
\end{enumerate}
\end{lemma}

See \cite[Proposition\,3.1]{MPP} for details about Lemma\,\ref{th:jgnbfgy}.

\begin{proof}[Proof of Theorem\,\ref{th:kjhhfy7dy}]

(1) follows from Lemma\,\ref{th:latbound}.

(2). By Corollary\,\ref{cor:hfhn4}, for the existence of $|T|$ it is enough to prove that, for every $x \in E$ the supremum $\sup D_x$ exists in $F$, where $D_x := \{T(u) - T(v): \, x = u \sqcup v\}$. Fix any $x \in E$. Then for every $t \in D_x$, say, $t = T(u) - T(v)$, where $x = u \sqcup v$, since $T$ preserves disjointness, one has by Lemma\,\ref{th:latbound}
$$
|t| = |T(u)| \sqcup |T(v)| \stackrel{\tiny {\rm Lemma}\,\ref{th:jgnbfgy}}{\sqsubseteq} |T(x)|.
$$
By Lemma\,\ref{th:jgnbfgy}\,(2), $|t| \le |T(x)|$. On the other hand, $T(x) \in D_x$ and $-T(x) \in D_x$. By the above, $\sup D_x = |T(x)|$. By Corollary\,\ref{cor:hfhn4}, $|T|$ exists and $\bigl|T \bigr| (x) = \bigl| T (x) \bigr|$ for all $x \in E$. The fact that $|T| \in \mathcal{P} (E,F)$ easily follows from the formula for $|T|$, see Corollary\,\ref{cor:hfhn4}.

(3). Since $T(x) = \bigl( T(x) \bigr)^+ - \bigl( T(x) \bigr)^-$, to prove the regularity of $T$, it is enough to show that the functions $\phi, \psi \colon E \to F$ defined by $\phi(x) = \bigl( T(x) \bigr)^+$ and $\psi(x) = \bigl( T(x) \bigr)^-$ for all $x \in E$ are OAOs. Indeed, fix any $x,y \in E$ with $x \perp y$. Then
$$
\phi(x+y) = \bigl( T(x \sqcup y) \bigr)^+ = \bigl( T(x) \sqcup T(y) \bigr)^+ = \bigl( T(x) \bigr)^+ \sqcup \bigl( T(y) \bigr)^+ = \phi(x) \sqcup \phi(y).
$$
So, $\phi$ is a disjointness preserving OAO. Similarly, $\psi$ is. It is left to show that $\phi = T \vee 0$ and $\psi = (-T) \vee 0$. Obviously, $T \le \phi$ and $0 \le \phi$. Let $S \in \mathcal{O}(E,F)$ satisfy $T \le S$ and $0 \le S$. Then for every $x \in E$ one has $\phi(x) = \bigl( T(x) \bigr)^+ \vee 0 \le S(x)$, that is, $\phi \le S$ and so $T^+ = \phi$. Analogously, $T^- = \psi$.

(4) Let $\{x,y\}$ be laterally bounded by $e \in E$, that is, $x \sqsubseteq e$ and $y \sqsubseteq e$. By Lemma\,\ref{th:latbound}, $T(x) \sqsubseteq T(e)$ and $T(y) \sqsubseteq T(e)$. By Lemma\,\ref{th:jgnbfgy}, $\bigl( T (x) \bigr)^+ \sqsubseteq \bigl( T (e) \bigr)^+$ and $\bigl( T (y) \bigr)^- \sqsubseteq \bigl( T (e) \bigr)^-$ and therefore $\bigl( T (x) \bigr)^+ \le \bigl( T (e) \bigr)^+$ and $\bigl( T (y) \bigr)^- \le \bigl( T (e) \bigr)^-$. Taking into account that $\bigl( T (e) \bigr)^+ \wedge \bigl( T (e) \bigr)^- = 0$, we obtain that $\bigl( T (x) \bigr)^+ \wedge \bigl( T (y) \bigr)^- = 0$.
\end{proof}

Now we come back to the following Meyer lemma, which precedes Meyer's theorem for linear operators.

\begin{lemma}[Lemma\,2.39, \cite{ABu}] \label{le:mayers}
Let $E, F$ be Riesz spaces with $F$ Archimedean. Then for every disjointness preserving linear operator $T \in \mathcal L_b (E,F)$ and every $x,y \in E^+$ one has $(Tx)^+ \wedge (Ty)^- = 0$.
\end{lemma}

The above Meyer's lemma is not longer true for OAOs due to the following simple example.

\begin{example}
Define a function $T \colon C[0,1] \to C[0,1]$ by setting $\mathbf{1} = \mathbf{1}_{[0,1]}$ and
$$
T(x) = \left\{
         \begin{array}{ll}
           \mathbf{1}, & \hbox{if} \,\, x = \mathbf{1}, \\
           - \mathbf{1}, & \hbox{if} \,\, x = 2 \cdot \mathbf{1}, \\
           0, & \hbox{if} \,\, x \in C[0,1] \setminus \{\mathbf{1}, 2 \cdot \mathbf{1}\}.
         \end{array}
       \right.
$$
Then $T$ is a disjointness preserving OAO and $\bigl( T (\mathbf{1}) \bigr)^+ \wedge \bigl( T (2 \cdot \mathbf{1}) \bigr)^- = \mathbf{1} \neq 0$.
\end{example}

There are also examples of the kind of OAOs acting between Dedekind complete Riesz spaces, however somewhat involved. For instance, define a map $S \colon L_p \to L_p$, $0 \le p \le \infty$ by setting $S(x) = x \pmb{\cap} \mathbf{1} - x \pmb{\cap} (2 \cdot \mathbf{1})$. By \cite[Theorem\,4]{KKP}, $S$ is an OAO. One can easily show that $|Sx| \le |x|$ for all $x \in L_p$, and hence $S$ preserves disjointness. Moreover, $\bigl( S (\mathbf{1}) \bigr)^+ \wedge \bigl( S (2 \cdot \mathbf{1}) \bigr)^- = \mathbf{1} \wedge (2 \cdot \mathbf{1}) = \mathbf{1} \neq 0$.

Such examples exist, because an OAO may have independent behavior on collinear vectors (see Proposition\,4.14 and Theorem\,6.1 of \cite{Po} for examples of OAOs with independent behavior on fragments of different elements of a Riesz space). So Meyer's lemma for OAOs is true for laterally bounded elements only, see Theorem\,\ref{th:kjhhfy7dy}\,(4).

\section{Remarks and open problems}

We do not know whether the sufficient condition in Theorem\,\ref{th:veekfuy} for the existence of $S \vee T$ is necessary for the case where the range space is not Dedekind complete.

\begin{prob} \label{pr:kkhdh}
Let $E, F$ be Riesz spaces and $S,T \in \mathcal{O}(E,F)$. Assume that $S \vee T$ exists in $\mathcal{O}(E,F)$. Does for every $x \in E$ there exists $\sup \{S(u) + T(v): \, x = u \sqcup v\}$ in $F$?
\end{prob}

Let $S \vee T$ exist in $\mathcal{O}(E,F)$, and let $\widehat{F}$ be a Dedekind completion of $F$ so that $F$ is an order dense ideal of $\widehat{F}$. Let $J \colon F \to \widehat{F}$ be the inclusion embedding. For every $P \in \mathcal{O}(E,F)$ we define an operator $\widehat{P} \in \mathcal{O}(E,\widehat{F})$ by setting $\widehat{P} = J \circ P$. Then the set $\{\widehat{S}(u) + \widehat{T}(v): \, x = u \sqcup v\}$ is order bounded in $F$ by $(S \vee T)(x)$ and by the Dedekind completeness of $\widehat{F}$, $R_1(x) := \sup \{\widehat{S}(u) + \widehat{T}(v): \, x = u \sqcup v\}$ exists in $\widehat{F}$. Then by Theorem\,\ref{th:veekfuy}, $\widehat{S} \vee \widehat{T}$ exists in $\mathcal{O}(E,\widehat{F})$ and equals $R_1$. Moreover, it is easily seen that $R_1 \le \widehat{S \vee T}$. One can show that, if $R_1 = \widehat{S \vee T}$ then $\sup \{S(u) + T(v): \, x = u \sqcup v\}$ exists in $F$ for all $x \in E$, and the answer to Problem\,\ref{pr:kkhdh} is affirmative.

\begin{prob}
Is $R_1 = \widehat{S \vee T}$ true in all cases?
\end{prob}

\section{Declarations}

\subsection{Funding}

The second-named author was supported by Weizmann Institute of Science (Israel) Emergency program for scientists affected by the war in Ukraine.

The third-named author was supported by the Ministry of Education of Ukraine Grant no 0122U000857.

\subsection{Availability of data and materials}

All items of the bibliography are available in accordance with publishing agreements.

\end{document}